\documentclass[11pt]{amsart}
\usepackage[usenames]{color}
\usepackage{amssymb}
\usepackage{latexsym}
\usepackage{graphics}
\usepackage{graphicx}
\usepackage{epstopdf}
\usepackage[all]{xy}
\usepackage{verbatim}

\usepackage{float}
\usepackage{placeins}
\input epsf

\def\TSG{{\mathrm{TSG}_+}}
\def\fix{{\mathrm{fix}}}
\def\Aut{{\mathrm{Aut}}}

\newtheorem{thm}{Theorem}

\newtheorem{lemma}{Lemma}

\newtheorem*{edge}{Edge Embedding Lemma}
\newtheorem*{subgroup}{Subgroup Lemma}
\newtheorem*{complete}{Complete Graph Theorem}
\newtheorem*{automorphism}{Automorphism Theorem}
\newtheorem*{realizability}{Realizability Lemma}
\newtheorem*{D2}{No-$D_2$ Lemma}
\newtheorem*{isometry}{Isometry Theorem}
\newtheorem*{DmSubgroups}{$D_m \times D_m$ Lemma}
\newtheorem*{A4}{$A_4$ Theorem}
\newtheorem*{S4}{$S_4$ Theorem}
\newtheorem*{A5}{$A_5$ Theorem}

\def\inv{{^{-1}}}

\def\Z{{\mathbb Z}}

\def\R{{\mathbb R}}

\def\a{{\alpha}}
\def\b{{\beta}}
\def\g{{\gamma}}
\def\G{{\Gamma}}

\def\f{{\phi}}

\def\s{{\sigma}}
\def\t{{\tau}}

\def\o{{\mathrm{order}}}

\def\TSG{{\mathrm{TSG_+}}}
\def\Aut{{\mathrm{Aut}}}
\def\Diff{{\mathrm{Diff_+}}}
\def\fix{{\mathrm{fix}}}

\newcommand{\x}{\times}

\newcommand{\orb}[1]{\langle #1 \rangle}

\def\so{{\mathrm{SO}}}
\def\gcd{{\mathrm{gcd}}}
\def\lcm{{\mathrm{lcm}}}

\begin{document}
\title[Classification]{Classification of Topological Symmetry Groups of $K_n$}
\author{Erica Flapan, Blake Mellor, Ramin Naimi, and Michael Yoshizawa}



\address{Department of Mathematics, Pomona College, Claremont, CA 91711, USA}

\address{Department of Mathematics, Loyola Marymount University, Los Angeles, CA 90045, USA}

\address{Department of Mathematics, Occidental College, Los Angeles, CA 90041, USA}

\address{Department of Mathematics, UC Santa Barbara, Santa Barbara, CA, USA}

\date \today

\thanks{This research was supported in part by NSF grants DMS-0905087, DMS-0905687 and DMS-0905300.}

\begin{abstract}

In this paper we complete the classification of topological symmetry groups for complete graphs $K_n$ by characterizing which $K_n$ can have a cyclic group, a dihedral group, or a subgroup of $D_m \times D_m$ for odd $m$, as its topological symmetry group.
Using this classification, one can algorithmically determine all possible topological symmetry groups of $K_n$ for any $n$.

\end{abstract}

\maketitle

\section{Introduction}\label{S:intro}
The symmetries of a molecule affect many of its properties, including its reactions, its crystallography, its spectroscopy, and its quantum chemistry.  Molecular symmetries are also an important tool in classifying molecules.  Chemists have long used the group of rotations, reflections, and reflections composed with rotations as a means of representing the symmetries of a molecule.  This group is known as the {\it point group} because its elements  fix a point of $\mathbb{R}^3$.  But using the point group to represent molecular symmetries only makes sense for rigid molecules.  In fact, while most small molecules are rigid, larger molecules can be flexible, and some molecules are rigid except for pieces that can rotate around specific bonds. On the left of   Figure~\ref{flexible} we illustrate a molecular M\"{o}bius ladder which is flexible because of its length, and on the right we illustrate a molecule that has pieces on either side that rotate around certain bonds.  
\begin{figure} [h]
\includegraphics[height=2.7cm]{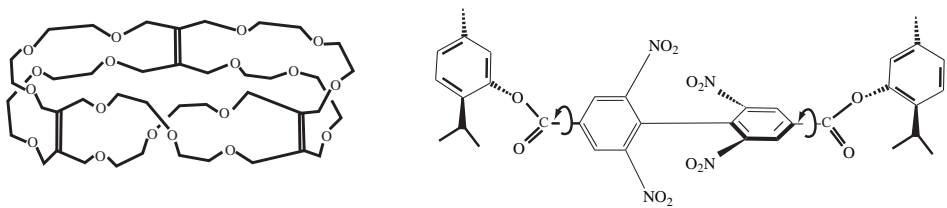}
\caption{The molecule on the left is flexible, while the one on the right has pieces that rotate around certain bonds. }
\label{flexible}
\end{figure}

The amount of rigidity a molecule has depends on its chemistry, not just on its length or its geometry.  Thus a purely mathematical characterization of all molecular symmetry groups is not possible.  The point group treats all molecules as if they are rigid objects.   We now take the opposite approach and consider molecules as flexible graphs embedded in $\mathbb{R}^3$.  In particular, we define the {\it topological symmetry group} of an embedded graph $\Gamma$ to be the group of automorphisms of $\Gamma$ that are induced by homeomorphisms of $\mathbb{R}^3$.  

 In order to illustrate the difference between the point group and the topological symmetry group, consider the molecular M\"{o}bius ladder illustrated on the left in Figure~\ref{flexible}.  Any symmetry of this molecule must take the set of three double bonds to itself.  The point group of the molecular M\"{o}bius ladder is $\mathbb{Z}_2$ because its only rigid symmetry is obtained by turning the graph over left to right.  Its topological symmetry includes this rigid symmetry as well as an order six automorphism of the graph induced by rotating the molecule by $120^{\circ}$ while slithering the half-twist back to its original position.  Thus the topological symmetry group is the dihedral group with 12 elements, which we denote by $D_6$.

Jon Simon \cite{si} first defined the
{\it topological symmetry group} in order to study symmetries of non-rigid molecules.   However this group can be used to characterize the topological symmetries of any graph embedded in $\mathbb{R}^3$ or $S^3$.  In fact, the study of  graphs embedded in $S^3$ can be thought of as an extension of knot theory, since a knot with vertices on it is an embedded graph.  Furthermore, given the extensive study of symmetries in knot theory, it is natural to consider symmetries of embedded graphs.

Formally, we introduce the following terminology. Let $\gamma$ be an abstract graph, and let $\Aut(\gamma)$ denote the automorphism group of $\gamma$.  Let $\Gamma$ be the image of an embedding of $\gamma$ in $S^3$.  The \emph{topological symmetry group} of $\Gamma$, denoted by $\mathrm{TSG}(\Gamma)$, is the subgroup of $\Aut(\gamma)$ which is induced by homeomorphisms of the pair $(S^3,\Gamma)$. The {\it orientation preserving topological symmetry group} of $\Gamma$, denoted by $\TSG(\Gamma)$, is the subgroup of $\Aut(\gamma)$ which is induced by orientation preserving homeomorphisms of the pair $(S^3,\Gamma)$. Observe that $\TSG(\Gamma)$ is either equal to $\mathrm{TSG}(\Gamma)$ or is a normal subgroup of $\mathrm{TSG}(\Gamma)$ of index two.  In this paper we are only concerned with $\TSG(\Gamma)$, and thus for simplicity we abuse notation and refer to the group $\TSG(\Gamma)$ simply as the {\it topological symmetry group} of $\Gamma$.

Frucht \cite{fr} showed that every finite group is the automorphism group of some connected graph.  Since every graph can be embedded in $S^3$, it is natural to ask whether every finite group can be  $\TSG(\Gamma)$ 
for some connected graph $\Gamma$ embedded in $S^3$.   Flapan, Naimi, Pommersheim, and Tamvakis \cite{fnpt} answered this question in the negative, proving that there are strong restrictions on which groups can occur as topological symmetry groups for some embedded graph.  For example, they prove that $\TSG(\Gamma)$ can never be the alternating group $A_n$ for $n>5$.  Furthermore, they show that for any 3-connected graph $\Gamma$, the group  $\TSG(\Gamma)$ is isomorphic to a finite subgroup of the group $\mathrm{Diff}_+(S^3)$ of orientation preserving diffeomorphisms of $S^3$.
Their proofs use the topological machinery of the Jaco-Shalen \cite{JS} and Johannson \cite{Jo} Characteristic Submanifold Theorem, together with Thurston's Hyperbolization  Theorem \cite{Th}, and Mostow's Rigidity Theorem \cite{Th}.

The special case of the topological symmetry groups of complete graphs  is interesting to consider because a complete graph $K_n$ has the largest automorphism group of any graph with $n$ vertices.  
Building on the result that any $\TSG(\Gamma)$ for an embedded 3-connected graph $\Gamma$ is isomorphic to a finite subgroup of $\mathrm{Diff}_+(S^3)$, Flapan, Naimi, and Tamvakis \cite{fnt} characterized which subgroups of $\mathrm{Diff}_+(S^3)$ can occur as $\TSG(\Gamma)$ for some embedding of a complete graph in $S^3$.  In particular they proved the following theorem.
\medskip \begin{complete} \cite{fnt}
\label{T:TSG2} 
A finite group $H$ is $\TSG(\Gamma)$ for some embedding $\Gamma$ of a complete graph in $S^3$ if and only if $H$ is isomorphic to a finite cyclic group, a dihedral group, $A_4, S_4, A_5$, or a finite subgroup of $D_m \times D_m$ for some odd $m$.
\end{complete}

As the first step in the proof of the above result, the authors show that every topological symmetry group of a complete graph can be induced by a finite group of orientation preserving isometries of $S^3$ for some embedding of the graph in $S^3$.  Then they consider two cases according to whether or not the group of isometries inducing $\TSG(\Gamma)$ preserves a standard Hopf fibration of $S^3$.  Finally, they give examples to demonstrate that each of the groups listed in the Complete Graph Theorem can actually occur as $\TSG(\Gamma)$ for some embedding $\Gamma$ of a complete graph in $S^3$.

While the Complete Graph Theorem tells us which groups can occur as topological symmetry groups of some complete graph in $S^3$, it does not tell us for a given group $G$ which complete graphs $K_n$ can have that group as its $\TSG(\Gamma)$ for some embedding $\Gamma$ of $K_n$ in $S^3$.  This is the question that we answer in this paper.  In the next section we will review the progress that has already been made on this problem, and discuss how the results of this paper complete the solution.

\section{Survey of the classification for complete graphs} \label{S:previous}

There has already been significant progress on the problem of determining which complete graphs can be embedded with a particular group as its topological symmetry groups.  In particular, Flapan, Mellor, and Naimi \cite{fmn2} proved the following theorems characterizing which complete graphs have an embedding $\Gamma$ such that $\TSG(\Gamma)$ is isomorphic to one of the {\it polyhedral} groups $A_4$, $S_4$ or $A_5$.

\begin{A4} \cite{fmn2}
A complete graph $K_m$ with $m\geq 4$ has an embedding $\Gamma$ in $S^3$ such that  $\TSG(\Gamma) \cong A_4$ if and only if $m \equiv 0$, $1$, $4$, $5$, $8 \pmod {12}$.
\end{A4}

\begin{A5} \cite{fmn2}
A complete graph $K_m$ with $m\geq 4$ has an embedding $\Gamma$ in $S^3$ such that $\TSG(\Gamma) \cong A_5$ if and only if $m \equiv 0$, $1$, $5$, $20 \pmod{60}$.
\end{A5}

\begin{S4} \cite{fmn2}
A complete graph $K_m$ with $m\geq 4$ has an embedding $\Gamma$ in $S^3$ such that  $\TSG(\Gamma) \cong S_4$ if and only if $m \equiv 0$, $4$, $8$, $12$, $20 \pmod {24}$.
\end{S4}

In the current paper, we finish the classification of which complete graphs can have a given group as their topological symmetry group by considering cyclic groups, dihedral groups, and subgroups of $D_m \times D_m$ (for $m$ odd).  The subgroups of $D_m \times D_m$ (for $m$ odd) are given below.

\begin{DmSubgroups} \cite{fmn1}
Let $m \geq 3$ be odd, and let $G \leq D_m \times D_m$. Then $G$ is isomorphic to one of the following groups for some  $r$ and $s$, both odd and at least 3: $\Z_2$, $\Z_r$, $\Z_{2r}$, $D_2$, $D_r$, $D_{2r}$, $\Z_r \times \Z_s$, $\Z_r \times D_s$, $D_r \times D_s$, or $(\Z_r \times \Z_s)\rtimes \Z_2$ (where, if $x \in \Z_r \times \Z_s$ and $y \in \Z_2$, $xy = yx\inv$).
\end{DmSubgroups}

In particular, we prove the following three theorems.

\begin{thm} \label{T:Dm}
Let $G = \Z_m$ or $D_m$.  A complete graph $K_n$, $n > 6$, has an embedding $\Gamma$ such that $\TSG(\Gamma) = G$ 
if and only if one of the following conditions holds:
\begin{enumerate}
	\item $m \geq 4$ is even, and $n \equiv 0 \pmod m$.
	\item $m \geq 3$ is odd and $n \equiv 0, 1, 2, 3 \pmod m$.
	\item $G= D_2$, and $n \equiv 0, 1, 2 \pmod 4$.
	\item $G = \Z_2$.
\end{enumerate}
\end{thm}

\begin{thm} \label{T:ZxZxZ}
Let $G =  \Z_r \x \Z_s$ or $(\Z_r \x \Z_s) \rtimes \Z_2$ where $r, s$ are odd and $\gcd(r,s) > 1$.  A complete graph $K_n$, with $n > 6$, has an embedding $\Gamma$ with $\TSG(\Gamma) = G$ if and only if one of the following conditions holds:
\begin{enumerate}
	\item $rs \vert n$.
	\item $\gcd(r,s) = 3$ and $rs \vert (n-3)$.
	\item $G = \Z_3 \x \Z_3$ and $9 \vert (n-6)$.
	\item $G = (\Z_3 \x \Z_3) \rtimes \Z_2$ and $18 \vert (n-6)$.
\end{enumerate}
\end{thm}

\begin{thm} \label{T:DxD}
Let $G =  \Z_r \x D_s$ or $D_r \times D_s$ where $r, s \geq 3$ are odd.  A complete graph $K_n$, with $n > 6$, has an embedding $\Gamma$ with $\TSG(\Gamma) = G$ if and only if one of the following conditions holds:
\begin{enumerate}
	\item $2rs \vert n$.
	\item $G = \Z_3 \x D_3$ and $18 \vert (n-6)$.
	\item $G = D_3 \times D_3$ and $36 \vert (n-6)$.
\end{enumerate}
\end{thm}

The proofs of Theorems \ref{T:Dm}, \ref{T:ZxZxZ} and \ref{T:DxD} use many of the methods and results developed to prove the theorems for $A_4$, $S_4$ and $A_5$.  For each group $G$, we need to show that the given restrictions are {\em necessary}, meaning that an embedding of $K_n$ realizing $G$ is possible only if $n$ satisfies the restrictions, and {\em sufficient}, meaning that we can construct an embedding of $K_n$ realizing $G$ whenever $n$ satisfies the restrictions.  These two aspects of the problem are usually treated separately.

\subsection{Necessity}

The first step in determining which restrictions are {\em necessary} on $K_n$ to realize a group $G$ is to establish restrictions on which automorphisms of $K_n$ can be induced by homeomorphisms of some embedding of the graph in $S^3$, and hence which elements of $\Aut(K_n)$ could possibly be elements of the topological symmetry group of some embedding of $K_n$.  This is done using the Automorphism Theorem \cite{fl2}.

\begin{automorphism} \cite{fl2}
Let $K_n$ be a complete graph on $n > 6$ vertices and let $\phi$ be an automorphism of $K_n$. Then there is an embedding $\Gamma$ of $K_n$ in $S^3$ such that $\phi$ is induced by an orientation preserving homeomorphism $h$ of $(S^3,\Gamma)$ of order $m$ if and only if the cycles and fixed vertices of $\phi$ can be described by one of the following:
\begin{enumerate}
	\item $m > 2$ is even, all cycles of $\phi$ are of order $m$, and $\phi$ fixes no vertices. 
	\item $m = 2$, all cycles of $\phi$ are of order $m$, and $\phi$ fixes at most two vertices. 
	\item $m$ is odd, all cycles of $\phi$ are of order $m$, and $\phi$ fixes at most three vertices. 
	\item $m$ is an odd multiple of 3 and $m > 3$, all cycles of $\phi$ are of order $m$ except one of order 3, and $\phi$ fixes no vertices.
\end{enumerate}
\end{automorphism}

\noindent{\sc Remark:}  The hypothesis that $n>6$ in the Automorphism Theorem is the reason for the same hypothesis in the statements of Theorems \ref{T:Dm}, \ref{T:ZxZxZ} and \ref{T:DxD}.  Chambers and Flapan \cite{ch} have characterized which groups can occur as $\TSG(\Gamma)$ or $\mathrm{TSG}(\Gamma)$ for each of the complete graphs with $n\leq 6$ vertices.
\medskip

The problem then becomes placing restrictions on possible {\em combinations} of automorphisms.  This is made much easier by the Isometry Theorem which follows from \cite{fnt} together with the Geometrization Theorem \cite{mf}.  This theorem allows us to assume that, roughly speaking, the topological symmetry group of an embedded complete graph is induced by isometries.

\begin{isometry}
Let $\Omega$ be an embedding of some $K_{n}$ in $S^3$.  Then $K_n$ can be re-embedded in $S^3$ as $\Gamma$ such that $\TSG(\Omega) \leq \TSG(\Gamma)$ and $\TSG(\Gamma)$ is induced by an isomorphic finite subgroup of $\so(4)$.
\end{isometry}
\medskip

Assuming that the elements of the topological symmetry group are isometries places significant restrictions on their fixed-point sets and how they intersect.  These, together with the conditions provided by the Automorphism Theorem, allow us to make geometric and combinatorial arguments to decide which complete graphs $K_n$ can be embedded with a given topological symmetry group $G$.

An important ingredient in many of these arguments is to use the No-$D_2$ Lemma \cite{cfo} to rule out topological symmetry groups containing $D_2$ for the graphs $K_{4r+3}$.

\begin{D2} \cite{cfo}
There is no embedding $\Gamma$ of $K_{4r+3}$ in $S^3$ such that $D_2 \leq \TSG(\Gamma)$.
\end{D2}

\subsection{Sufficiency}

To prove that the restrictions are sufficient, we need to prove that when they are satisfied, it is possible to realize the group $G$ as the topological symmetry group of $K_n$.  This is done in several stages.  First, we identify a group of isometries isomorphic to $G$ (since the Isometry Theorem assures us that if $G$ can be realized, then it can be realized using isometries), and embed the vertices of the graph so that the set of vertices is invariant under the action of $G$.  We then wish to extend the embedding to the edges of the graph so that the entire graph is invariant under $G$; this is relatively simple if the action of $G$ does not fix any of the vertices of the graph; otherwise, various conditions must be met.  These conditions are given by the Edge Embedding Lemma \cite{fmn2}.

\begin{edge} \cite{fmn2} 
Let $G$ be a finite subgroup of $\Diff(S^3)$, and let $\g$ be a graph whose vertices are embedded in $S^3$ as a set $V$ which is invariant under $G$ such that $G$ induces a faithful action on $\g$.  Suppose that adjacent pairs of vertices in $V$ satisfy the following hypotheses:
\begin{enumerate}
	\item If a pair is pointwise fixed by non-trivial elements $h, g \in G$, then $\fix(h) = \fix(g)$.
	\item For each pair $\{v, w\}$ in the fixed point set $C$ of some non-trivial element of $G$, there is an arc $A_{vw} \subseteq C$ bounded by $\{v,w\}$ whose interior is disjoint from $V$ and from any other such arc $A_{v'w'}$ 
	\item If a point in the interior of some $A_{vw}$ or a pair $\{v,w\}$ bounding some $A_{vw}$ is setwise invariant under an $f \in G$, then $f(A_{vw}) = A_{vw}$.
	\item If a pair is interchanged by some $g \in G$, then the subgraph of $\g$ whose vertices are pointwise fixed by $g$ can be embedded in a proper subset of a circle.
	\item If a pair is interchanged by some $g \in G$, then $\fix(g)$ is non-empty, and $\fix(h) \neq \fix(g)$ if $h \neq g$.
\end{enumerate}
Then the embedding of the vertices of $\g$ can be extended to the edges of $\g$ in $S^3$ such that the resulting embedding of $\g$ is setwise invariant under $G$.
\end{edge}

Much of the work in the proofs of Theorems \ref{T:Dm}, \ref{T:ZxZxZ} and \ref{T:DxD} lies in verifying that the conditions of the Edge Embedding Lemma are satisfied.  Once the embedding is extended to the edges, we have an embedding whose topological symmetry group {\it contains} $G$, but it may be larger.  The final step is to add local knots to the edges of the embedding so as to reduce the topological symmetry group to $G$.  There are some subtle topological issues that must be dealt with to show that this can be done; these were addressed first in the Knot Embedding Lemma \cite{fnpt} and further in the Knot Addition Lemma \cite{fmn1}.  The proofs of these results use the machinery of pared manifolds and the Characteristic Submanifold Theorem previously mentioned in Section \ref{S:intro}.

The case when the vertices can be embedded so there are no fixed vertices was solved first by Flapan, Naimi and Tamvakis \cite{fnt}, giving the Realizability Lemma (which depended on the earlier Knot Embedding Lemma).

 \begin{realizability} \cite{fnt}
{\it Let $\gamma $ be a $3$-connected graph, and let $H \leq\Aut(\gamma )$
be such that no vertex is fixed by any non-trivial element of $H$.
Suppose that $H$ is isomorphic to a subgroup $G\leq \so(4)$ such that every involution of $G$ has non-empty fixed point set and if $g\in G$ is an involution then no $h\in G$ distinct from $g$ has $\fix(h)=\fix(g)$. Then there is an embedding $\Gamma$ of $\gamma$ in $S^3$ such that
$H=\TSG(\Gamma )$ and $H$ is induced by $G$.}
 \end{realizability}

When the vertices of the graph are {\em not} fixed, we combine the Edge Embedding Lemma with the Subgroup Lemma \cite{fmn1} (proven using the Knot Addition Lemma).  This allows us to embed the graph with a topological symmetry group that may be too large, but then reduce the topological symmetry group as desired.

\begin{subgroup} \cite{fmn1}
Suppose $\Gamma$ is an embedding of a graph in $S^3$ such that 
either 
(i)~$\Gamma$ is 3-connected and 
contains an edge $e$ which is not pointwise fixed by any non-trivial element of $\TSG(\Gamma)$;
or
(ii)~$\Gamma$ is isomorphic to $K_n$ for some $n > 6$ 
and  $\TSG(\Gamma)$ is a finite cyclic group, a dihedral group, or a subgroup of $D_m \times D_m$ for some odd $m$. 
Then for every (possibly trivial) subgroup $H$ of $\TSG(\Gamma)$, 
there is an embedding $\Gamma'$ of $K_n$ such that $H = \TSG(\Gamma')$.
\end{subgroup}

\subsection{Examples}

Combining Theorems \ref{T:Dm}, \ref{T:ZxZxZ} and \ref{T:DxD} with the results of \cite{ch} and \cite{fmn2} enables one to determine all possible (orientation-preserving) topological symmetry groups of a given $K_n$ completely algorithmically.  For example, Table~\ref{T:examples} lists, for each $n= 2, 3, \cdots, 20$, every (orientation-preserving) topological symmetry group that can occur for some embedding of a complete graph $K_n$ .

\begin{table}[t]
\begin{tabular}{|c|l|l|l|l|}
\hline
Graph & \parbox{.75in}{Polyhedral Groups} & $\Z_m$ and $D_m$ & \parbox{1in}{\rule{0pt}{.2in}$\Z_r \x \Z_s$ and \\$(\Z_r \x \Z_s) \rtimes \Z_2$\smallskip} & \parbox{.8in}{\rule{0pt}{.2in}$\Z_r \x D_s$ and \\$D_r \x D_s$\smallskip} \\ \hline
$K_2$ & None & $\Z_2$ & None & None \\ \hline
$K_3$ & None & \parbox{2in}{\rule{0pt}{.2in}$\Z_3$, $D_3$\smallskip} & None & None \\ \hline
$K_4$ & $A_4$, $S_4$ & \parbox{2in}{\rule{0pt}{.2in}$\Z_2$, $\Z_3$, $\Z_4$, $D_2$, $D_3$, $D_4$\smallskip} & None & None \\ \hline
$K_5$ & $A_4$, $A_5$ & \parbox{2in}{\rule{0pt}{.2in}$\Z_2$, $\Z_3$, $\Z_5$, $D_2$, $D_3$, $D_5$\smallskip} & None & None \\ \hline
$K_6$ & None & \parbox{2in}{\rule{0pt}{.2in}$\Z_2$, $\Z_3$, $\Z_5$, $\Z_6$, $D_2$, $D_3$, $D_5$, $D_6$\smallskip} & \parbox{1in}{\rule{0pt}{.2in}$\Z_3 \x \Z_3$, \\$(\Z_3 \x \Z_3) \rtimes \Z_2$} & \parbox{.8in}{\rule{0pt}{.2in}$\Z_3 \x D_3$, \\$D_3 \x D_3$} \\ \hline
$K_7$ & None & \parbox{2in}{\rule{0pt}{.2in}$\Z_2$, $\Z_3$, $\Z_5$, $\Z_7$, $D_3$, $D_5$, $D_7$\smallskip} & None & None \\ \hline
$K_8$ & $A_4$, $S_4$ & \parbox{2in}{\rule{0pt}{.2in}$\Z_2$, $\Z_3$, $\Z_4$, $\Z_5$, $\Z_7$, $\Z_8$, \\$D_2$, $D_3$, $D_4$, $D_5$, $D_7$, $D_8$\smallskip} & None & None \\ \hline
$K_9$ & None & \parbox{2in}{\rule{0pt}{.2in}$\Z_2$, $\Z_3$, $\Z_7$, $\Z_9$, $D_2$, $D_3$, $D_7$, $D_9$\smallskip} & \parbox{1in}{\rule{0pt}{.2in}$\Z_3 \x \Z_3$,\\$(\Z_3 \x \Z_3) \rtimes \Z_2$\smallskip}  & None \\ \hline
$K_{10}$ & None & \parbox{2in}{\rule{0pt}{.2in}$\Z_2$, $\Z_3$, $\Z_5$, $\Z_7$, $\Z_9$, $\Z_{10}$, \\$D_2$, $D_3$, $D_5$, $D_7$, $D_9$, $D_{10}$\smallskip} & None & None \\ \hline
$K_{11}$ & None & \parbox{2in}{\rule{0pt}{.2in}$\Z_2$, $\Z_3$, $\Z_5$, $\Z_9$, $\Z_{11}$, \\$D_3$, $D_5$, $D_9$, $D_{11}$\smallskip} & None & None \\ \hline
$K_{12}$ & $A_4$, $S_4$ & \parbox{2.3in}{\rule{0pt}{.2in}$\Z_2$, $\Z_3$, $\Z_4$, $\Z_5$, $\Z_6$, $\Z_9$, $\Z_{11}$, $\Z_{12}$, \\$D_2$, $D_3$, $D_4$, $D_5$, $D_6$, $D_9$, $D_{11}$, $D_{12}$\smallskip} & \parbox{1in}{\rule{0pt}{.2in}$\Z_3 \x \Z_3$,\\$(\Z_3 \x \Z_3) \rtimes \Z_2$\smallskip} & None \\ \hline
$K_{13}$ & $A_4$ & \parbox{2in}{\rule{0pt}{.2in}$\Z_2$, $\Z_3$, $\Z_5$, $\Z_{11}$, $\Z_{13}$, \\$D_2$, $D_3$, $D_5$, $D_{11}$, $D_{13}$\smallskip} & None & None \\ \hline
$K_{14}$ & None & \parbox{2in}{\rule{0pt}{.2in}$\Z_2$, $\Z_3$, $\Z_7$, $\Z_{11}$, $\Z_{13}$, $\Z_{14}$, \\$D_2$, $D_3$, $D_7$, $D_{11}$, $D_{13}$, $D_{14}$\smallskip} & None & None \\ \hline
$K_{15}$ & None & \parbox{2in}{\rule{0pt}{.2in}$\Z_2$, $\Z_3$, $\Z_5$, $\Z_7$, $\Z_{13}$, $\Z_{15}$, \\$D_3$, $D_5$, $D_7$, $D_{13}$, $D_{15}$\smallskip} & \parbox{1in}{\rule{0pt}{.2in}$\Z_3 \x \Z_3$\smallskip} & None \\ \hline
$K_{16}$ & $A_4$ & \parbox{2in}{\rule{0pt}{.2in}$\Z_2$, $\Z_3$, $\Z_4$, $\Z_5$, $\Z_7$, $\Z_8$, $\Z_{13}$, $\Z_{15}$, $\Z_{16}$, $D_2$, $D_3$, $D_4$, $D_5$, $D_7$, $D_8$, $D_{13}$, $D_{15}$, $D_{16}$\smallskip} & None & None \\ \hline
$K_{17}$ & $A_4$ & \parbox{2in}{\rule{0pt}{.2in}$\Z_2$, $\Z_3$, $\Z_5$, $\Z_7$, $\Z_{15}$, $\Z_{17}$, \\$D_2$, $D_3$, $D_5$, $D_7$, $D_{15}$, $D_{17}$\smallskip} & None & None \\ \hline
$K_{18}$ & None & \parbox{2in}{\rule{0pt}{.2in}$\Z_2$, $\Z_3$, $\Z_5$, $\Z_6$, $\Z_9$, $\Z_{15}$, $\Z_{17}$, $\Z_{18}$, $D_2$, $D_3$, $D_5$, $D_6$, $D_9$, $D_{15}$, $D_{17}$, $D_{18}$\smallskip} & \parbox{1in}{\rule{0pt}{.2in}$\Z_3 \x \Z_3$,\\$(\Z_3 \x \Z_3) \rtimes \Z_2$\smallskip} & \parbox{.75in}{\rule{0pt}{.2in}$\Z_3 \x D_3$, \\$D_3 \x D_3$\smallskip} \\ \hline
$K_{19}$ & None & \parbox{2in}{\rule{0pt}{.2in}$\Z_2$, $\Z_3$, $\Z_9$, $\Z_{17}$, $\Z_{19}$, \\$D_3$, $D_9$, $D_{17}$, $D_{19}$\smallskip} & None & None \\ \hline
$K_{20}$ & $A_4$, $S_4$, $A_5$ & \parbox{2in}{\rule{0pt}{.2in}$\Z_2$, $\Z_3$, $\Z_4$, $\Z_5$, $\Z_9$, $\Z_{10}$, $\Z_{17}$, $\Z_{19}$, $\Z_{20}$, $D_2$, $D_3$, $D_4$, $D_5$, $D_9$, $D_{10}$, $D_{17}$, $D_{19}$, $D_{20}$\smallskip} & None & None \\ \hline
\end{tabular}
\medskip
\caption{Topological symmetry groups for $K_n$ with $2 \leq n \leq 20$.} \label{T:examples}
\end{table}

As a final more elaborate application of our results, we give the following list of all possible groups that can be $\TSG(\Gamma)$ for embeddings of the complete graph $K_{140}$ in $S^3$.  Note that $140 = 4 \x 5 \x 7$: \begin{itemize}
	\item $A_4$, $S_4$, $A_5$
	\item $\Z_2$, $\Z_3$, $\Z_4$, $\Z_5$, $\Z_7$, $\Z_{10}$, $\Z_{14}$, $\Z_{20}$, $\Z_{23}$, $\Z_{28}$, $\Z_{35}$, $\Z_{69}$, $\Z_{70}$, $\Z_{137}$, $\Z_{139}$, $\Z_{140}$
	\item $D_2$, $D_3$, $D_4$, $D_5$, $D_7$, $D_{10}$, $D_{14}$, $D_{20}$, $D_{23}$, $D_{28}$, $D_{35}$, $D_{69}$, $D_{70}$, $D_{137}$, $D_{139}$, $D_{140}$
	\item $\Z_5 \x D_7$, $\Z_7 \x D_5$, $D_5 \x D_7$
\end{itemize}
\noindent{\sc Remark:} Since $\gcd(5,7) = 1$, Theorem \ref{T:ZxZxZ} gives no additional groups.
\bigskip

The remainder of the paper is devoted to the proofs of Theorems \ref{T:Dm}, \ref{T:ZxZxZ} and \ref{T:DxD}.  In Section \ref{S:Dm} we  prove Theorem \ref{T:Dm}, in Section \ref{ZxZxZ}, we prove Theorem \ref{T:ZxZxZ}, and in Section \ref{S:necessary} we prove Theorem  \ref{T:DxD}.  

\FloatBarrier

\section{Proof of Theorem \ref{T:Dm}} \label{S:Dm}

Theorem \ref{T:Dm} describes when a complete graph has an embedding whose topological symmetry group is either cyclic or dihedral.  Throughout, the notation $\orb{S}_G$ denotes the orbit of a set of points $S$ under the action of a group $G$.

\begin{proof}
Since this proof is somewhat technical, we will divide it into the following steps: \begin{itemize}
	\item[{\bf Step 1:}] Prove the {\em necessity} of the conditions.
	\item[{\bf Step 2:}] Show that part (4) follows from parts (1) and (2).
	\item[{\bf Step 3:}] Show that under the conditions of part (1), $D_m \le \TSG(\Gamma)$ for some embedding $\Gamma$ of $K_n$.
	\item[{\bf Step 4:}] Show that under the conditions of part (2), $D_m \le \TSG(\Gamma)$ for some embedding $\Gamma$ of $K_n$.
	\item[{\bf Step 5:}] Show that under the conditions of part (3), $D_m \le \TSG(\Gamma)$ for some embedding $\Gamma$ of $K_n$.
	\item[{\bf Step 6:}]  Show that if $D_m \le \TSG(\Gamma)$ for some embedding $\Gamma$ of $K_n$, then or $G = \Z_m$ or $D_m$, $\TSG(\Gamma') = G$ for some embedding $\Gamma'$ of $K_n$.
\end{itemize}
\bigskip

\subsection*{\bf Step 1: Necessity.}
Since $n > 6$, and $\Z_m$ and $D_m$ each contain an element of order $m$, the necessity of the conditions for $m \geq 3$ is immediate from the Automorphism Theorem.  If $m = 2$, the restriction that $n \not\equiv 3 \pmod 4$ for $D_2$ comes from the No-$D_2$ Lemma.

\subsection*{\bf Step 2: $\TSG = \Z_2$.}
Given any $n >6$,
we let $m=n$, so $n \equiv 0 \pmod m$, 
and hence  $D_m \le \TSG(\G)$
for some embedding $\Gamma$ of $K_n$, by parts (1) and (2).
Since $\Z_2 \le D_m$, it follows from the Subgroup Lemma that 
$\TSG(\G') = \Z_2$
for some embedding $\G'$ of $K_n$, proving part (4).

\subsection*{\bf Step 3: $m \geq 4$ is even.}
Suppose that $m \ge 4$ is even, and $n = mr$ for some $r$.
We describe a subgroup of $\so(4)$ isomorphic to $D_m$.
Let $h_1$ be a rotation of $S^3$ of order $m$ with fixed point set a geodesic circle $C_1$,
and $h_2$ be a rotation of $S^3$ of order $m/2$ with fixed point set a geodesic circle $C_2$ that is setwise fixed by $h_1$.
Let $h$ be the glide rotation $h_1 \circ h_2$ of order $m$,
and let $a$ be an involution of $S^3$ whose fixed point set $A$ intersects each of $C_1$ and $C_2$  orthogonally in two points.  Then the group $H = \orb{h, a} $ generated by $h$ and $a$ is isomorphic to  $D_m$.  

To embed the vertices of $K_n$, 
choose $\lfloor r/2 \rfloor$ points $v_1, \dots, v_{\lfloor r/2 \rfloor}$ in 
$S^3 - (C_1 \cup C_2 \cup \orb{A}_H)$
which have disjoint orbits under $H$, 
and embed $2m\lfloor r/2 \rfloor$ vertices as the orbits of the $v_i$'s under $H$.  
If $r$ is even, we've embedded all the vertices; 
if $r$ is odd, then we have $m$ vertices that have not yet been embedded.  
Embed one of these vertices as some $v\in A - (C_1 \cup C_2)$, 
and embed the remaining vertices as the points in $\orb{v}_H$.
Note that $h ^{m/2}$ is an involution that fixes $A$ setwise but not pointwise;
thus, there will be two vertices embedded on each of the $m/2$ circles in $\orb{A}_H$.

We have embedded the vertices of $K_n$ as a set $V$ 
which is invariant under $H$ so that $H$ induces a faithful action on $V$. 
In the case where $r$ is even, we embedded all vertices in $S^3 - (C_1 \cup C_2 \cup \orb{A}_H)$,
so no vertex is fixed by a nontrivial element of $H$. 
Hence, by the Realizability Lemma, we obtain an embedding $\Gamma$ of $K_n$ such that
$\TSG(\Gamma) = D_m$, as desired.
In the case where $r$ is odd,
we use the Edge Embedding Lemma to embed the edges of $K_n$;
so we need to verify that the five conditions of the lemma are satisfied.
We will use the following claim to verify one of the conditions.

\medskip
\noindent{\sc Claim:}
For all $g \in H$,
$g(A)$ either equals $A$ or is disjoint from $A$.

{\em Proof of Claim:}
Let $g \in H$ be such that $g(A)$ intersects $A$.
We wish to show $g(A) = A$.
Let $x \in A \cap g(A)$.
Then $x$ is fixed by both $a$ and $gag\inv$.
Now, $agag\inv = h^i$ for some positive $i \le m$
(since, in any dihedral group, the product of two reflections is not a reflection).
So $x$ is fixed by $h^i $.
Now, $h^i = h_1^i h_2^i$ has fixed points
only if $i = m/2$ or $i = m$;
and, in either case, $h^i(A) = A$.
Since $agag\inv = h^i$ and $a\inv = a$, we get $ag\inv = g\inv ah^i$;
so $ag\inv(A) = g\inv ah^i (A) = g\inv(A)$,
i.e.,  $g\inv(A)$ is fixed by $a$.
Therefore $g\inv(A) \subset A$, which implies $g\inv(A) = A$
since $A$ is a circle and $g$ is an isometry.
Therefore $g(A) = A$, completing the proof of the claim. $\Box$

\medskip

Let $w = h^{m/2}(v)$; then $v$ and $w$ are the two vertices on the circle $A$.
Now, any pair of vertices that are pointwise fixed by a nontrivial element of $H$
are of the form $\{g(v),g(w)\}$ for some $g \in H$;
and, for each such pair, $gag\inv$ is the unique nontrivial element of $H$ that fixes the pair pointwise.
So condition~(1) of the Edge Embedding Lemma is satisfied.
Let $A_{vw}$ be one of the two arcs on $A$ from $v$ to $w$.
For each pair of vertices $\{v',w'\} = \{g(v),g(w)\}$,
let $A_{v'w'}=g(A_{vw})$.
Any two such arcs are disjoint since, by the above claim, for all $g \in H$,
$g(A)$ either equals $A$ or is disjoint from $A$.
Hence condition~(2) is satisfied as well.
The only nontrivial elements of $H$ that
fix a point in the interior of $A_{vw}$
or setwise fix the pair of vertices $\{v, w\}$
are $h^{m/2}$ and $a$,
both of which setwise fix $A_{vw}$.
And similarly for other pairs $\{v', w'\}$;
hence condition~(3) is satisfied.
The pair $\{v,w\}$ is interchanged by only one element of $H$, namely $h^{m/2}$;
and this element pointwise fixes the circle $C_1$ and no vertices.
So, by a similar argument for the other pairs $\{v', w'\}$,
we see that conditions~(4) and (5) are also satisfied.  Hence, by the Edge Embedding Lemma, there is an embedding $\Gamma$ of $K_n$ such that $D_m \leq \TSG(\Gamma)$.

\subsection*{\bf Step 4: $m \geq 3$ is odd.}
Now, suppose $m \ge 3$ is odd, 
and $n = mr + k$ for some $r$ and for some $k\in \{0,1,2,3\}$ (where $k \ne 3$ if $m = 3$). 
This time, we let $h$ be a rotation of $S^3$ of order $m$ with fixed point set a circle $C$, and let $a$ be an involution of $S^3$ whose fixed point set $A$ intersects $C$ orthogonally in two points.  
Then $H = \orb{h, a} = D_m$.  
We embed $k$ vertices on $C$ as follows:  if $k = 1$, one vertex is embedded as a point of $C \cap A$; if $k = 2$, one vertex is embedded in $C - A$, and the other is embedded as the image of the first under $a$; if $k = 3$, we embed two vertices as for $k = 2$, and the third as for $k = 1$.  
For the remaining $n-k = mr$ vertices, we proceed as follows.
Embed $\lfloor r/2 \rfloor$ points in the complement of $C \cup A$ such that they have distinct orbits,
giving a total of $2m \lfloor r/2 \rfloor$ points under the action of $H$. 
If $r$ is even, this allows us to embed all of the $mr$ remaining vertices, as desired.
If $r$ is odd, we can embed only $mr -m$ vertices this way. 
The remaining $m$ vertices are embedded as the orbit of a point in $A - C$.
Note that in this case, since $m$ is odd, 
for each circle $A' \subset \orb{A}_H$,
only one vertex is on $A' - C$.

Thus, for part~(2) of the theorem, we have embedded the vertices of $K_n$ as a set $V$ which is invariant under $H$ so that $H$ induces a faithful action on $V$.  We now use the Edge Embedding Lemma to embed the edges of $K_n$.  The fixed point sets of the elements of $H$ consist of $C$ and the circles in $\orb{A}_H$.
Any two of these circles meet at the two points $C \cap A$.  
Since we have placed a vertex at at most one of these points, two elements of $H$ can fix a pair of vertices only if they share the same fixed point set.  So condition (1) of the Edge Embedding Lemma is satisfied.  

Now, if $v$ and $w$ are distinct vertices pointwise fixed by an element of $H$,
then either $\{v,w\} \subset C$ or $\{v,w\} \subset A'$ for some $A' \in \orb{A}_H$.
Since $k \leq 3$, the vertices on $C$ can all be connected 
by arcs in $C - V$ whose interiors are disjoint.  
The vertex in $C\cap A$ (if there is one) 
can be connected to each of the vertices (if any) in $\orb{A-C}_H$  
by arcs in $\orb{A-C}_H$ whose interiors are all disjoint from each other and from the arcs on $C$.  
So condition (2) is satisfied.  And any element of $H$ which fixes a point in the interior of one of these arcs, or interchanges its endpoints, also fixes the arc setwise.  So condition (3) is satisfied.  
The only elements of $H$ which interchange vertices are involutions.  
Each involution pointwise fixes at most two vertices of $K_n$ and the edge between them, 
so condition (4) is satisfied.  
Finally, the fixed point set of each involution is non-empty, 
and is not the fixed point set of any other element of $H$, so condition (5) is satisfied.  
Hence, by the Edge Embedding Lemma, there is an embedding $\Gamma$ of the edges and vertices of $K_n$ such that $D_m \leq \TSG(\Gamma)$.

\subsection*{\bf Step 5: $m = 2$.}
Finally, suppose 
$m = 2$ and $n = 4r+k$ for some $r$, and for some $k \in \{0,1,2\}$.
 We use the same construction as for part~(2),
 with $H = \orb{h, a}$, $C$, and $A$ as before.
 The vertices are embedded as in part~(2),
 i.e., if $k=1$, embed one vertex as a point in $C \cap A$;
 if $k=2$, embed two vertices as a point in $C -A$ and its image under $a$;
 and embed the remaining $4r$ vertices as the orbit under $H$ of $r$ points in the complement of $C \cup A$. 
The conditions of the Edge Embedding Lemma are verified as in the construction for part~(2),
except that they are simpler here since 
$\orb{A}_H = \{A\}$ and there are no vertices on $A - C$.
Hence, again there is an embedding $\Gamma$ of $K_n$ such that $D_m \leq \TSG(\Gamma)$.

\subsection*{\bf Step 6: Conclusion.}
We have shown that for all values of $m$ and $n$ in the hypothesis of the theorem, 
there is an embedding $\Gamma$ of $K_n$ such that $H= D_m \leq \TSG(\Gamma)$. 
In all of our constructions of $\Gamma$,
either (i)~there is a vertex which is not fixed by any non-trivial element of $H$,
or (ii)~there is a vertex which is fixed by an involution in $H$ but not by every element of $H$:
case~(i) occurs in parts~(1) and (2) if $n = mr+k$ and $r \ge 2$, 
and also in part~(3) since $n = 4r+k \ge 7$ and $k \le 2$, which implies $r \ge2$;
case~(ii) occurs in parts~(1) and (2) if $n = mr+k$ and $r = 1$.
Let $G= \Z_m$ or $D_m$.
In case~(i), by the Subgroup Lemma, 
we can re-embed $K_n$ as $\Gamma'$ such that $\TSG(\Gamma')=G$.
In case~(ii), by \cite[Lemma 2.2]{fmn2}, 
$\TSG(\Gamma)$ does not contain $A_4$ and hence is not a polyhedral group. 
So, by the Complete Graph Theorem and the Subgroup Lemma, 
again there is a re-embedding $\Gamma'$ of $K_n$ such that $\TSG(\Gamma') = G$.
\end{proof}

\section{Proof of Theorem \ref{T:ZxZxZ}}\label{ZxZxZ}

Theorem \ref{T:ZxZxZ} characterizes when the topological symmetry group can be a product of cyclic groups.  We begin by defining a {\it standard presentation} for these groups:

\begin{itemize}
	\item $\orb{\rho,\s}$ is a {\em standard presentation} of $\Z_r \x \Z_s$ if $\Z_r \x \Z_s=\orb{\rho,\s}$, where $\o(\rho)=r$, $\o(\s)=s$, and $\rho$ commutes with $\s$.
	\item $\orb{\rho,\s,\phi}$ is a {\em standard presentation} of $(\Z_r \x \Z_s) \rtimes \Z_2$ if $(\Z_r \x \Z_s) \rtimes \Z_2 =\orb{\rho,\s,\phi}$, where $\o(\rho)=r$, $\o(\s)=s$, $\o(\phi)=2$,  $\rho\s = \s\rho$, $\phi\rho = \rho\inv \phi$, and $\phi\s = \s\inv \phi$.
\end{itemize}

\medskip

Lemma \ref{L:n=15} shows that the additional restriction for $(\Z_r \x \Z_s) \rtimes \Z_2$ in part (4) of Theorem \ref{T:ZxZxZ} is necessary.

\begin{lemma} \label{L:n=15}
Assume that $n > 6$ and $K_n$ has an embedding $\Gamma$ such that $(\Z_3 \x \Z_3) \rtimes \Z_2 \leq \TSG(\Gamma)$.  
If $9 \vert (n-6)$, then $18\vert (n-6)$.
\end{lemma}

\begin{proof}
Assume $9 \vert (n-6)$, but $n$ is odd.  By the Isometry Theorem, we may assume that $(\Z_3 \x \Z_3) \rtimes \Z_2$ is generated by a finite group of orientation-preserving isometries of $(S^3, \Gamma)$.  Let $\orb{\rho, \s, \f}$ be a standard presentation of $(\Z_3 \x \Z_3) \rtimes \Z_2$.  Observe that, since $n$ is odd, $\f$ must fix an odd number of vertices; by the Automorphism Theorem, this means $\f$ fixes exactly one vertex.  Also, it follows from \cite{cfo} that $\rho$ and $\s$ each fix either 0 or 3 vertices, $\rho(\fix(\s))=\fix(\s)$, $\s(\fix(\rho)) = \fix(\rho)$, and the fixed point sets are disjoint.  It also follows from \cite{cfo} that there are at most two disjoint sets of three vertices which are setwise invariant under both $\rho$ and $\s$ (including the fixed vertices).  Let $X$ be the union of all such sets; then $\vert X \vert = 0, 3$ or $6$, and includes the fixed vertices of $\rho$ and $\s$.  We prove below, as in \cite{cfo}, that $\vert X \vert = 6$.  Let $\GammaÕ$ be the result of removing the vertices of $X$ (and all adjacent edges) from the graph $\Gamma$.  Then, by the Automorphism Theorem, $\rho$ partitions the vertices of $\GammaÕ$ into 3-cycles, as does $\s$.  But $\s$ does not fix any of the cycles of $\rho$ (since we've removed the vertices of $X$), so it permutes the 3-cycles of $\rho$
(since $\rho$ and $\s$ commute, $\s$ takes each cycle of $\rho$ to some cycle of $\rho$). 
Hence the number of 3-cycles must be a multiple of 3, so 9 divides the number of vertices of $\Gamma'$, which is $n - \vert X\vert$.  Since $9\vert (n-6)$, this means $\vert X\vert = 6$, so there are exactly two disjoint sets, $V = \{v_0, v_1, v_2\}$ and $W = \{w_0, w_1, w_2\}$, which are setwise invariant under both $\rho$ and $\s$.

Now there are three possibilities (up to renaming $\rho$, $\s$, $V$ and $W$):  \begin{enumerate}
	\item $\rho$ fixes $V$ pointwise and permutes $W$, and $\s$ fixes $W$ pointwise and permutes $V$, or
	\item $\rho$ fixes $V$ pointwise and permutes $W$, and $\s$ permutes both $V$ and $W$, or
	\item both $\rho$ and $\s$ permute both $V$ and $W$.
\end{enumerate}

\subsection*{\bf Case 1:}  Assume $\rho$ fixes $V$ pointwise and permutes $W$, and $\s$ fixes $W$ pointwise and permutes $V$.  Then $\rho\f(v_i) = \f\rho\inv(v_i) = \f(v_i)$, so $\f(v_i)$ is a fixed vertex of $\rho$, and hence $\f(v_i) \in V$.  So $\f$ fixes $V$ setwise.  Since $\f$ has order 2 and $\vert V\vert = 3$, $\f$ must fix an odd number of vertices in $V$.  So, by the Automorphism Theorem, $\f$ fixes exactly one vertex of $V$.  Similarly, using $\s$, $\f$ fixes exactly one vertex of $W$.  But then $\f$ fixes two vertices, which contradicts our earlier observation that it must fix exactly one vertex.

\subsection*{\bf Case 2:}  Assume $\rho$ fixes $V$ pointwise and permutes $W$, and $\s$ permutes both $V$ and $W$.  As in Case 1, $\f$ must fix $V$ setwise, and fixes one vertex of $V$.  Let $W' = \f(W)$, with $w_i' = \f(w_i)$.  Then $\rho(w_i') = \rho\f(w_i) = \f\rho\inv(w_i) = \f(w_j) = w_j'$, so $\rho$ fixes $W'$ setwise.  Similarly, $\s$ fixes $W'$ setwise.  But $V$ and $W$ are the only sets of three vertices which are setwise invariant under both $\rho$ and $\s$.  $W' \neq V$, since $\f(W') = W$ and $\f(V) = V$.  So we must have $W' = W$, and hence $\f$ fixes $W$ setwise.  But then we get a contradiction as in Case 1.

\subsection*{\bf Case 3:}  Assume both $\rho$ and $\s$ permute both $V$ and $W$.  As in Case 2, $\f(V)$ is either $V$ or $W$, and $\f(W)$ is either $V$ or $W$.  So either $\f(V) = V$ and $\f(W) = W$, or $\f(V) = W$ and $\f(W) = V$.  The first case leads to a contradiction as in Case 1, so we may assume that $\f(V) = W$ and $\f(W) = V$.  Without loss of generality, say $\f(v_0) = w_0$.

Without loss of generality (replacing $\s$ by $\s\inv$ if needed), we may assume that $\rho$ and $\s$ have the same action on $V$.  But then $\rho \neq \s$ on $W$ (or $\rho\s\inv$ would fix 6 vertices, and hence be trivial by Smith Theory), and neither $\rho$ nor $\s$ fix $W$ pointwise, so $\rho = \s\inv$ on $W$ since $\rho$ and $\s$ both have order~3.  
So, without loss of generality, $\rho(v_i) = \s(v_i) = v_{i+1}$, $\rho(w_i) = w_{i+1}$ and $\s(w_i) = w_{i-1}$, where the indices are calculated modulo 3.

Then $\f(v_1) = \f\rho(v_0) = \rho\inv\f(v_0) = \rho\inv(w_0) = w_2$.  But, also, $\f(v_1) = \f\s(v_0) = \s\inv\f(v_0) = \s\inv(w_0) = w_1$.  This is a contradiction.

\subsection*{} So every case leads to a contradiction, and we conclude that $n$ must be even, so $18 \vert (n-6)$.
\end{proof}

\bigskip

\noindent We can now prove Theorem \ref{T:ZxZxZ}.

\begin{proof}
Since $\gcd(r,s) > 1$, we can rewrite $\Z_r \x \Z_s$ as $\Z_{\gcd(r,s)} \x \Z_{\lcm(r,s)}$, where the first factor is not trivial.  Also, since $r$ and $s$ are both odd, $\gcd(r,s)$ and $\lcm(r,s)$ are also odd.  Therefore, by redefining $r$ and $s$ if necessary, we may assume $r \vert s$.  

\subsection*{\bf Result for $\Z_r \x \Z_s$.}  The result for $\Z_r \x \Z_s$ follows from \cite{cfo} (including part (3)).  

\subsection*{\bf Necessity for $(\Z_r \x \Z_s) \rtimes \Z_2$.}  Since $\Z_r \x \Z_s$ is a subgroup of $(\Z_r \x \Z_s) \rtimes \Z_2$, the necessity of most of the conditions for $(\Z_r \x \Z_s) \rtimes \Z_2$ follows; the condition in part (4) is required by Lemma \ref{L:n=15}.

\subsection*{\bf Sufficiency for $(\Z_r \x \Z_s) \rtimes \Z_2$ when $rs \vert n$.}
We consider the subgroup of SO(4) isomorphic to $(\Z_r \x \Z_s) \rtimes \Z_2$ generated as follows.  Let $\a$ be a rotation of order $r$ about a geodesic circle $A$, $\b$ be a rotation of order $s$ about a geodesic circle $B$ disjoint from $A$ which is setwise fixed by $\a$, and $\f$ be a rotation of order 2 about a geodesic circle $C$ which intersects each of $A$ and $B$ in two points.  Then $G = \orb{\a,\b,\f} \cong (\Z_r \x \Z_s) \rtimes \Z_2$.  

Suppose $rs \vert n$.  If $n = (2k)rs$ for some positive integer $k$, we pick $k$ points in $S^3$ which are not fixed by any non-trivial elements of $G$, and whose orbits under $G$ are disjoint.  Then each orbit has $2rs$ elements, and we embed the vertices of $K_n$ as the resulting $2rsk$ points.  Since none of the vertices is fixed by any nontrivial element of $G$, the hypotheses of the Realizability Lemma are satisfied.

Now consider the case when $n = (2k+1)rs$.  Embed $(2k)rs$ vertices as described above.  For the remaining vertices, pick a point $v$ on $C - (A\cup B)$, so $\f(v) = v$; then the orbit of $v$ under $G$ contains $rs$ points.  Embed the remaining vertices as these points.  
Since $n$ is odd, by the Automorphism Theorem, 
any involution fixes only one vertex.
So $v$ is the only vertex in $\orb{v}_G$ which lies on $C$.  Therefore, at most one vertex of $\orb{v}_G$ lies on the fixed point set of any nontrivial element of $G$, so the first four conditions of the Edge Embedding Lemma are trivially satisfied.  
Since $G$ has no even-order elements other than involutions,
pairs of vertices are only interchanged by involutions, which each have distinct fixed point sets homeomorphic to $S^1$, 
so condition (5) is satisfied.  

So, if $rs \vert n$, we have an embedding $\Gamma$ of $K_n$ such that $(\Z_r \x \Z_s) \rtimes \Z_2 \leq \TSG(\Gamma)$.

\subsection*{\bf Sufficiency for $(\Z_r \x \Z_s) \rtimes \Z_2$ when $rs \vert (n-3)$.}  Suppose that $n = krs + 3$ and $\gcd(r,s) = 3$.  Since we assume that $r\vert s$, this means $r = 3$.  
We embed $krs$ vertices as in the previous paragraphs (for $k$ even or odd).  
Now let $w$ be one point of $B \cap C$, so $w$ is fixed by both $\b$ and $\f$ (if $k$ is odd, then we choose $w$ to be the point of $B \cap C$ which is in the same component of $C - A$ as the vertex $v$ in the last paragraph).  Then the orbit of $w$ under $G$ contains 3 points, all of which are in $B$; we embed our last three vertices as these three points.  Each involution of $G$ fixes one of these three points and interchanges the other two.  At most one vertex lies in the intersection of any two fixed point sets of elements of $G$, so condition (1) of the Edge Embedding Lemma is satisfied.  Each pair of vertices fixed by an involution are joined by a unique arc in the fixed point set of the involution whose interior is disjoint from the other vertices and fixed point sets; and the three vertices on $B$ are joined by disjoint arcs in $B$, so condition (2) is satisfied.  The three arcs on $B$ are fixed by $\b$ and permuted cyclically by $\a$.  Each involution of $G$ fixes one of the vertices on $B$ and interchanges the other two, so it interchanges two of the three arcs, and fixes the third arc setwise (while reversing its endpoints).  Therefore, the arcs on $B$ satisfy condition (3); the arcs in the fixed point sets of the involutions are disjoint from the other vertices and fixed points sets, so they also satisfy condition (3).  Only involutions interchange pairs of vertices, and each involution fixes at most two vertices of the graph pointwise, so condition (4) is satisfied.  Finally, the fixed point sets of the involutions are all circles, and are distinct from the fixed point sets of all other elements of the group, so condition (5) is satisfied.  So we have an embedding $\Gamma$ of $K_n$ such that $(\Z_r \x \Z_s) \rtimes \Z_2 \leq \TSG(\Gamma)$.

\subsection*{\bf Sufficiency for $(\Z_3 \x \Z_3) \rtimes \Z_2$ when $18 \vert (n-6)$}
Finally, we consider when $n = 18k+6$ and $r = s= 3$, which is part~(4) of the theorem.  We embed the vertices as above, except now there are three vertices on each of $A$ and $B$,
and there is no $v$ as above since $2k$ is even.
Then, the conditions of the Edge Embedding Lemma are satisfied as above, and we again have an embedding $\Gamma$ of $K_n$ such that $(\Z_3 \x \Z_3) \rtimes \Z_2 \leq \TSG(\Gamma)$.

\subsection*{\bf Result for $(\Z_r \x \Z_s) \rtimes \Z_2$.}
So in each case we have an embedding $\Gamma$ of $K_n$ such that $(\Z_r \x \Z_s) \rtimes \Z_2 \leq \TSG(\Gamma)$.  Since $(\Z_r \x \Z_s) \rtimes \Z_2$ is not a subgroup of $\so(3)$, the Complete Graph Theorem tells us that $\TSG(\Gamma)$ is a subgroup of $D_m \times D_m$ for some odd $m$.  Then, by the Subgroup Lemma, there is an embedding $\Gamma'$ of $K_n$ such that $\TSG(\Gamma') = (\Z_r \x \Z_s) \rtimes \Z_2$.
\end{proof}

\section{Proof of Theorem \ref{T:DxD}} \label{S:necessary}

Theorem \ref{T:DxD} characterizes when the topological symmetry group can be a product of a cyclic and dihedral group, or a product of two dihedral groups.  As in the last section, we begin with the definition of a {\bf standard presentation} of these products: \begin{itemize}
	\item $\orb{\rho,\s,\b}$ is a {\em standard presentation} of $\Z_r \x D_s$ if $\Z_r \x D_s=\orb{\rho,\s,\b}$, where $\o(\rho)=r$, $\o(\s)=s$, $\o(\b)=2$,  $\rho$ commutes with $\b$ and $\s$, and $\b\s = \s\inv \b$.
	\item $\orb{\rho,\a,\s,\b}$ is a {\em standard presentation}  of $D_r \x D_s$ if $D_r \x D_s=\orb{\rho,\a,\s,\b}$, where $\o(\rho)=r$, $\o(\s)=s$, $\o(\a)=\o(\b)=2$, each of $\rho$ and $\a$ commutes with $\s$ and $\b$, $\a \rho = \rho\inv \a$, and $\b \s = \s\inv \b $.
\end{itemize}

\medskip

We first prove some necessary conditions for realizing $\Z_r \times D_s$ and $D_r \times D_s$.  These conditions will depend on some elementary facts about isometries of $S^3$.

\begin{lemma} \label{fact-PreserveFixedPointSet}
Let $\t$ and $\b$ be bijective maps from a set to itself such that $\b \t = \t^{\pm 1} \b$. Then $\b (\fix(\t)) = \fix(\t)$.
\end{lemma}
\begin{proof}
If $z \in \fix(\t)$, then $\b(z) = \b\t(z) = \t^{\pm 1}\b(z)$.  Therefore, $\b(z) \in \fix(\t)$, so $\b(\fix(\t)) \subseteq \fix(\t)$.  
By a similar argument, $\b\inv(\fix(\t)) \subseteq \fix(\t)$.  Since $\b$ is a bijection, this means $\fix(\t) \subseteq \b(\fix(\t))$, and therefore $\b (\fix(\t)) = \fix(\t)$.
\end{proof}

\begin{lemma}  \label{fact-PreserveOrientation}
Let $G=\langle \t,\b\rangle$ be a finite subgroup of $\mathrm{SO}(4)$ which leaves a simple closed curve $C \subset S^3$ setwise invariant.  Suppose $\o(\t)>2$. Then $\t$ and $\b$ commute if and only if $\b$ preserves the orientation of $C$.
\end{lemma}

\begin{proof}  Since $G$ leaves $C$ setwise invariant, there is a neighborhood of $C$ which $G$ leaves setwise invariant and the action of $G$ on this neighborhood is conjugate to the action on the normal bundle of $C$.  Now since the normal bundle of a simple closed curve has a canonical trivialization by parallel transport, there is an invariant product neighborhood $C\times D^2$ whose product structure is preserved by $G$. Thus we can write the restrictions $\t|_{C\times D^2}$ and $\beta|_{C\times D^2}$ as product maps $(\t_1,\t_2)$ and $(\b_1,\b_2)$ respectively, where $\t_1$ and $\b_1$ are maps of $C$ and $\t_2$ and $\b_2$ are maps of $D^2$.

We make the following observations about the elements of any finite group action of a circle or disk:  
\begin{enumerate}
	\item If an element is orientation reversing, then it has order 2.  
	\item Any pair of orientation preserving elements commute.
	\item Suppose that $g$ is an element with order greater than 2.  Then any element which commutes with $g$ is orientation preserving.
\end{enumerate}

Since $\o(\t)>2$, for some $k = 1$ or 2, $\o(\t_k)>2$.  Hence by Observation 1, $\t_k$ must be orientation preserving.  Since $\t$ is orientation preserving, it follows that both $\t_1$ and $\t_2$ must be orientation preserving.

Suppose that $\b$ preserves the orientation of $C$, then so does $\b_1$.  
Since $\b$ is orientation preserving, it follows that $\b_2$ is also orientation preserving.  By Observation 2, $\t_j$ and $\b_j$ commute for each $j$.  Now $\t\b\t^{-1}\b^{-1}$ has finite order and pointwise fixes $C\times D^2$, so it must be trivial.
Hence $\t$ and $\b$ commute.

Conversely, suppose that $\t$ and $\b$ commute. 
Then for each $j$, $\t_j$ and $\b_j$ commute.  Recall that $k$ was chosen so that $\o(\t_k)>2$.  Since $\t_k$ and $\b_k$ commute, by Observation 3, $\b_k$ is orientation preserving.
Again since $\b$ is  orientation preserving, this means that both $\b_1$ and $\b_2$ are as well.  In particular, $\b$ preserves the orientation of $C$.
\end{proof}

\begin{lemma} \label{fact-TwoCircles}
Any nontrivial, orientation preserving, odd order rotation of $S^3$ setwise fixes exactly two geodesic circles.
\end{lemma}
\begin{proof}
Let $s$ be an orientation preserving, odd order rotation of  of $S^3$;
then $s  = s'|_{S^3}$ for some odd order element $s' \in \so(4)$.
Hence $s'$ is conjugate to a rotation $r$ that fixes 
the $x_1 x_2$-plane ($P$) setwise, 
and the $x_3 x_4$-plane ($P'$) pointwise. 
Since $r$ and $s'$ are conjugate,
it suffices to show that $r$ does not setwise fix any geodesic circle
other than the two in $P$ and $P'$.
Every geodesic circle of $S^3$ is 
the intersection of $S^3$ with a plane through the origin in $\R^4$. 
Suppose $r$ setwise fixes a geodesic $P'' \cap S^3$
distinct from $P \cap S^3$ and $P' \cap S^3$.
As $r$ is a linear transformation of $\R^4$, 
it must also setwise fix the plane $P''$.
Since $P$ and $P'$ are orthogonal complements of each other
and $P''$ is distinct from $P'$, 
$P''$ does not lie in the orthogonal complement of $P$;
hence $P''$ contains a vector $v$ whose projection onto $P$ is nonzero. 
So we can write $v = w + w'$, 
where $w$ and $w'$ are the projections of $v$ onto  $P$ and  $ P'$ respectively, 
and $w$ is nonzero. 
Now, $r(v) \in r(P'') = P''$, and 
$P''$ is a subspace, so $r(v) - v \in P''$. 
Furthermore, $r(w') = w'$ since $r$ pointwise fixes $P'$.
So $u = r(v) - v = r(w) - w$ is a nonzero vector in $P$ 
since $w \ne 0$ and $r$ rotates $w$ by a nonzero angle. 
Now, as $u$ lies in both $P$ and $P''$, 
$r(u)$ also lies in both $P$ and $P''$.
Hence $r(u)$ is parallel to $u$; 
so $r$ must rotate $u$ by a multiple of $\pi$, 
which contradicts the fact that $s$ has odd order. 
So $r$ does not setwise fix any geodesic circle other than 
$P \cap S^3$ and $P' \cap S^3$.
\end{proof}
\bigskip

We use these lemmas to prove the following result about subgroups of $\so(4)$ isomorphic to $\Z_r \times D_s$.

\begin{lemma} \label{lemma-ZrxDs-isometries}
Let $ \langle \rho,\s,\b \rangle $ be  a standard presentation of 
$\Z_r \x D_s \le \mathrm{SO}(4)$, where $r, s \ge 3$ and $r$ is odd.
Then for every positive $i<\o(\rho)$ and positive $j<\o(\s)$, $\rho^i $ and $\s^j$ are fixed point free.
\end{lemma}

\begin{proof} 
Suppose toward contradiction that
$\fix(\s^j)$ is nonempty for some positive $j < s$.
Then $\fix(\s^j)$ is a circle $C$.
Since $\b \s = \s \inv \b$,
by Lemma~\ref{fact-PreserveFixedPointSet}, $\b(C) = C$.
Furthermore, since $s \ge 3$, by Lemma~\ref{fact-PreserveOrientation}, 
$\b$ must reverse the orientation of $C$.
So $\b$ fixes exactly two points in $C$ and is a rotation by $\pi$
about a circle $B$ that intersects $C$
in those two points.

Since $\rho$ commutes with both $\b$ and $\s^j$,
by Lemma~\ref{fact-PreserveFixedPointSet} it setwise fixes each of $B$ and $C$.
Hence it setwise fixes $B \cap C$, 
which consists of two points.
As $\rho$ has odd order,
it cannot exchange these two points,
so it must fix each of them.
Now, for each circle $B$ and $C$,
since $\rho$ setwise fixes the circle
and fixes two points on it
and has odd order, 
it must pointwise fix the circle, so $\rho$ fixes $B \cup C$.
This implies $\rho$ is trivial,
which is a contradiction.

Now, suppose toward contradiction that
$\fix(\rho^j) $ is nonempty for some positive $j < r$.
Then $\fix(\rho^j)$ is a circle.
Since $\b$ and $\s$ each commute with $\rho^j$,
by Facts~1 and 2
they each setwise fix and preserve the orientation of $\fix(\rho^j)$.
But, by Lemma~\ref{fact-PreserveOrientation},  this contradicts the hypothesis that $\b \s = \s\inv \b$.
\end{proof}
\bigskip

We now apply these results to prove restrictions on the complete graphs which can be embedded so that their topological symmetry group contains $\Z_r \times D_s$ or $D_r \times D_s$.

\begin{lemma} \label{lemma-ZrxDs-nonrealizable}
Let  $n > 6$ and let $r,s\geq 3$ be odd.  Suppose $\Gamma$ is an embedding of $K_n$
such that $ \Z_r \x D_s$ is induced on $\Gamma$ by an isomorphic subgroup $G\leq \mathrm{SO}(4)$. 
\begin{enumerate}
\item
If some non-trivial element of  $G$ fixes a vertex of $\Gamma$, then $r=s=3$ and $2rs \vert (n-6)$. Otherwise, $2rs\vert n$.
\item
If $\orb{\rho, \s, \b}$ is a  standard presentation  of $G$,
then there exist geodesic circles $A$ and $B$ in $S^3$
such that: 
(i)~$G(A \cup B) = A \cup B$,
(ii)~$\b$ interchanges $A$ with $B$,
and 
(iii)~if a non-trivial element of  $G$ fixes a vertex of $\Gamma$, 
then each of $A$ and $B$ contains exactly three vertices of $\G$.
\item
If $H = \orb{\rho, \a, \s, \b} \le \so(4)$ is a standard presentation of $D_r \x D_s$ 
and induces an isomorphic subgroup on $\G$,
then,
in addition to the conclusions of part (2) above, we have
 $H(A \cup B) = A \cup B$
and $\a$ interchanges $A$ with $B$.
\end{enumerate}
\end{lemma}

\begin{proof} 
By  \cite[Lemma~2]{fnt},
$G$ satisfies the Involution Condition (defined in \cite{fnt}).
Therefore, by \cite[Proposition~3]{fnt}, 
$G$ preserves a standard Hopf fibration of $S^3$,
since otherwise it would be isomorphic to one of the polyhedral groups.
And by \cite[Lemma~3]{fnt}, 
there exist two fibers $A$ and $B$
such that $\{A,B\}$ is setwise invariant under $G$.
Let $\orb{\rho, \s, \b}$ be a  standard presentation  of $G$.
Since $\rho$ and $\s$ each have odd order,
they cannot interchange $A$ with $B$.
Therefore they must setwise fix each of $A$ and $B$.

Suppose $\b$ fixes $A$ setwise.
Then, since $\b$ commutes with $\rho$,
by Lemma~\ref{fact-PreserveOrientation},
$\b$ must preserve the orientation of $A$.
On the other hand,
$\b \s = \s\inv \b$, and $s \ge 3$;
so by Lemma~\ref{fact-PreserveOrientation},
$\b$ must reverse the orientation of $A$.
Thus we get a contradiction.
Since $\b(A \cup B) = A \cup B$ and $\b(A) \ne A$,
$\b$ must interchange $A$ with $B$.

We claim no involution in $G$ can fix any vertices.
Suppose toward contradiction that an involution $\g$ fixes a vertex $v_0$.
Then $\fix(\g)$ is a circle $X$.
By hypothesis, $\rho$ commutes with every element of $G$,
and in particular with $\g$.
Hence, by Lemma~\ref{fact-PreserveFixedPointSet},
$\rho(X) = X$,
which gives $\langle v_0 \rangle_{\rho} \subset X$.
By Lemma~\ref {lemma-ZrxDs-isometries}, for every positive $i < r$, $\rho^i$ has no fixed points, so $|\langle v_0 \rangle_{\rho}| = r \ge 3$.
But, by the Automorphism Theorem, $\g$ cannot fix more than two vertices, which is a contradiction.

If no vertex of $\G$ is fixed by any nontrivial element of $G$,
then every vertex orbit under $G$ has size $|G|=2rs$, 
and $2rs \vert n$.
Suppose a vertex $v_0$ is fixed by some nontrivial element $g_0 \in G$
that is not an involution.
Since $\fix(g_0)$ is nonempty, it is a circle $C$.
By replacing $g_0$ by its square if necessary,
we can assume $g_0 =  \rho^i \s^j$ for some $i,j$.
So $g_0$ commutes with both $\rho$ and $\s$.
Hence, by Lemma~\ref{fact-PreserveFixedPointSet}, $\orb{\rho, \s}$ setwise fixes $C$.
By Lemma~\ref {lemma-ZrxDs-isometries}, for every positive $k<r$ and $l<s$,
$\rho^k$ and $\s^l$ have no fixed points;
so $\orb{v_0}_{\orb{\rho, \s}} \subset C$
contains at least three vertices.
On the other hand, as $g_0$ is nontrivial, 
by the Automorphism Theorem
its fixed point set $C$ cannot contain more than 3 vertices.
It follows that $\orb{v_0}_{\orb{\rho, \s}} $
consists of exactly 3 vertices.
Thus, $\rho^3$ and $\s^3$ must each fix $v_0$.
Therefore, by Lemma~\ref{lemma-ZrxDs-isometries}, $r=s=3$.

Now, $g_0= \rho^i \s^j$ is a rotation of odd order about $C$.
So, by Lemma~\ref{fact-TwoCircles},
it setwise fixes at most two geodesic circles.
As $g_0 $ setwise fixes each of $A$ and $B$,
we must have $C=A$ or $C = B$.
Thus $v_0$ is in $A$ or $B$.
Recall that the orbit of $v_0$ under $\orb{\rho,\s}$
contains exactly three vertices;
and $\b$ interchanges $A$ with $B$.
It follows that $|\langle v_0 \rangle_G | = 6$.
Furthermore, these 6 vertices are the only vertices
that are fixed by any nontrivial element of $G$
since any other such vertex 
would also have an orbit of size 6
with 3 on $A$ and 3 on $B$,
contradicting that $C$ contains only 3 vertices. 
Thus every vertex that's not in $\orb{v_0}_G$
is fixed only by the trivial element 
and hence has an orbit of size $|G| = 2rs$.
Therefore $2rs \vert (n - 6)$, which finishes part~(1) of the lemma.
The above also establishes part~(2) of the lemma.

To prove part~(3),
suppose $\orb{\rho, \a, \s, \b}$
is a standard presentation of $D_r \x D_s$
and is induced by an isomorphic subgroup $H \le \mathrm{SO}(4)$ on $\G$.
But applying the argument given above for $\orb{\rho, \s, \b}$
to $\orb{\rho, \a, \s}$ instead,
we see that $\a$ also interchanges $A$ with $B$,
and hence $H(A \cup B) = A \cup B$.
\end{proof}
\bigskip

Finally, we prove restrictions on when the topological symmetry group can be $D_3 \times D_3$.

\begin{lemma} \label{lemma-DrDsNonrealiz}
Assume that $K_n$, with $n > 6$, has an embedding $\Gamma$ such that $D_3 \times D_3 \leq \TSG(\Gamma)$.  
If $18 \vert (n-6)$, then $36\vert (n-6)$.
\end{lemma}

\begin{proof}
Suppose toward contradiction that there is an embedding $\G$ of $K_n$
such that  $\TSG(\Gamma) =D_3 \x D_3$, and $n-6$ is an odd multiple of $18$.
Since $n > 6$, $K_n$ is 3-connected;
so $\TSG(\G)$ is
induced by an isomorphic subgroup $H$ of $\so(4)$.
Let $H = \orb{\rho, \a, \s, \b}$ be a standard presentation of 
$D_3 \x D_3$.

By Lemma~\ref{lemma-ZrxDs-nonrealizable},
there exist geodesic circles $A$ and $B$
such that $H(A \cup B) = A \cup B$
and each of these two circles contains exactly three vertices of $\G$.
Thus these six vertices are setwise invariant under $H$.
Let $\G'$ be the embedding of $K_{n'}$, where $n' = n-6$, 
obtained by removing these six vertices from $\G$.
Then $\G'$ is invariant under $H$, 
and hence under its subgroup  $\Z_r \x D_s = \orb{\rho,\s,\b}$.
Since  $18 \vert n'$, by Lemma~\ref{lemma-ZrxDs-nonrealizable}, 
no element of $\orb{\rho,\s,\b}$ 
fixes any vertex of $\G'$.
Similarly, no element of $\orb{\rho,\a,\s}$ 
fixes any vertex of $\G'$.
It follows that only elements of $H -  (\orb{\rho,\s,\b} \cup \orb{\rho,\a,\s})$
can possibly fix any vertex of $\G'$.

Now, some vertex in $\G'$ must be fixed by some nontrivial element in $H$,
since otherwise the orbit size of every vertex would be a multiple of  $|H|=36$,
contradicting that $n'$ is an odd multiple of $18$.
Let $\f = \a \b$.
Then every element of  $H -  (\orb{\rho,\s,\b} \cup \orb{\rho,\a,\s})$
is of the form $\rho^i \s^j \a \b$ for some $i$ and $j$.
Note that $(\rho^{2i}\s^{2j})\a\b(\rho^{2i}\s^{2j})\inv = \rho^{4i}\s^{4j}\a\b = \rho^i \s^j \a \b$.
So every element of  $H -  (\orb{\rho,\s,\b} \cup \orb{\rho,\a,\s})$
is an involution conjugate to $\f = \a\b$.
Thus, some element conjugate to $\f$ fixes a vertex of $\G'$.
Therefore $\f$ itself fixes some vertex $v \in \G'$.

By Lemma~\ref{lemma-ZrxDs-nonrealizable},
$\a$ and $\b$ each interchange $A$ with $B$.
It follows that $\f$ setwise fixes each of $A$ and $B$.
Since each of $A$ and $B$ contains three of 
the vertices of $\G - \G'$ and $\o(\f)=2$,
$\f$ must fix one vertex on each of $A$ and $B$.
These two vertices are distinct from $v$
since the latter is in $\G'$ while the former are in $\G - \G'$.
Therefore $\f$ fixes three vertices of $\G$, 
which is a contradiction since by the Automorphism Theorem 
no involution has three fixed points.
\end{proof}

With these lemmas in hand, we prove Theorem \ref{T:DxD}.

\begin{proof}
We first observe that the necessity of most of the conditions is given by Lemma \ref{lemma-ZrxDs-nonrealizable}; the last condition is required by Lemma \ref{lemma-DrDsNonrealiz}.  It remains to prove the conditions are sufficient.

\subsection*{\bf A realization of $D_r \x D_s$.}
Let $A$ and $B$ be the intersections of $S^3$ with 
the $x_1 x_2$- and the $x_3 x_4$-planes, respectively, in $\R^4$.
Let $\rho$ be a glide rotation obtained by composing
a rotation of $2\pi /r$ about $A$ with a rotation of $2\pi /r$ about $B$,
$\s$ a glide rotation obtained by composing a rotation of $2\pi /s$ about $A$ with 
a rotation of $-2\pi /s$ about $B$,
$T$ the geodesic torus that separates $A$ and $B$,
$\a$ a rotation by $\pi$ about a $(1,-1)$ curve on $T$, 
and 
$\b$ a rotation by $\pi$ about a $(1,1)$ curve on $T$.
Then $\o(\rho)=r$, $\o(\s)=s$, $\o(\a)=\o(\b)=2$, 
$\a$ and $\b$ each interchange $A$ with $B$,
each of $\rho$ and $\a$ commutes with each of $\s$ and $\b$,
$\a$ anticommutes with $\rho$, 
and $\b$ anti-commutes with $\s$.
Thus $G=\orb{ \rho,\a,\s,\b}$ is isomorphic to $D_r \x D_s$.

\subsection*{\bf Sufficiency when $2rs \vert n$.}  If $2rs \vert n$, then
$n = k(4rs)$ or $n = k(4rs) + 2rs$ for some $k \ge 0$.
In the former case,
we pick $k$ points $x_1, \cdots, x_k$ 
disjoint from the fixed point set of all nontrivial elements of $G$
such that any two such points have disjoint orbits under $G$.
We embed the vertices of $K_n$ as the points in the orbits of $x_1, \cdots, x_k$.
Since no vertex is fixed by any nontrivial element of $G$,
by the Realizability Lemma, 
we get an embedding $\Gamma$ of $K_n$ such that
$\TSG(\Gamma) = D_r \x D_s$.
We will refer to this set of $k(4rs)$ embedded vertices as $X$.

Now suppose $n = k(4rs) + 2rs$.
Let $\f = \a \b$.
Then $\f(A)=A$, $\f(B)=B$,
and $\f$ is a rotation by $\pi$ about a geodesic circle $C$
that intersects each of $A$ and $B$ in two points.
Thus $C - (A \cup B)$ consists of four arcs, $C_1, C_2, C_3, C_4$.
From our construction we can see that
the circle $\fix(\a)$ intersects exactly two of these arcs,
say $C_1$ and $C_3$.
In fact, $\a$ setwise fixes each of $C_1$ and $C_3$ 
while reversing their orientations,
and it interchanges $C_2$ with $C_4$.

Let $v \in C_1$ be a point
that is not on the fixed point set of any nontrivial element of $G$ other than $\f$.
Then $v$ is fixed only by $\f$ and the identity;
hence  $|\orb{v}_G| = |G|/ 2= 2rs$.
We embed $2rs$ vertices of $K_n$ as the orbit of $v$,
and the remaining $k(4rs)$ vertices as $X$.

Observe that every element of $G$ of even order
is an involution conjugate to $\a$, $\b$, or $\f$.
Since no vertex is embedded in $\fix(\a)$ or $\fix(\b)$,
no element conjugate to $\a$ or $\b$ fixes any vertices.
Since no vertex is embedded in $A$ or $B$,
no element of odd order fixes any vertices.
Hence only elements conjugate to $\f$ can fix vertices.

We see as follows that $\f$ fixes exactly two vertices.
Every vertex fixed by $\f$ must by in $\orb{v}_G $
since the $k(4rs)$ vertices not in the orbit of $v$ are not fixed by any nontrivial elements.
Now, $\b|_C =(\a \f) |_C = \a |_C$; hence $\a(v) = \b(v)$.
Since the only nontrivial elements of $G$
that take $v$ to a point in $C$ are $\f$, $\a$, and $\b$,
we see that $\orb{v}_G \cap C$ consists of exactly two points,
$v$ and $w=\a(v)$.
It follows that each element conjugate to $\f$ also fixes exactly two vertices.

We now verify that the conditions of the Edge Embedding Lemma are satisfied.
Condition~(1) is satisfied since only elements conjugate to $\f$ fix any vertices 
and no two such elements have the same fixed point set.
Recall that  $\a(C_1)=C_1$, 
and we chose $v$ to be in $C_1$.
Hence $w$ is also in $C_1$.
To satisfy Condition~(2),
let $A_{vw}$ be the arc in $C_1$ from $v$ to $w$.
Each involution $ g \f g\inv$ conjugate to $\f$ fixes exactly two vertices,
$v' = g(v)$ and $w' = g(w)$;
we let $A_{v'w'} = g(A_{vw})$.
Then the interior of any such arc $A_{v'w'}$ is disjoint from the set of all embedded vertices
as well as from any other such arc.
Condition~(3) is satisfied since
the only elements that fix a point in the interior of $A_{v'w'}$
or setwise fix $\{v', w'\}$ are $g \f g\inv$, $g \a g\inv$, and $g \b g\inv$,
which all setwise fix $A_{v'w'}$.
The pair $\{v', w'\}$ is interchanged only by $g \a g\inv$ and $g \b g\inv$,
which do not fix any vertices;
this implies Condition~(4) is satisfied.
Finally, since $g \a g\inv$ and $g \b g\inv$
have nonempty, distinct fixed point sets, 
namely $g(\fix(\a))$ and $g(\fix(\b))$,
Condition~(5) is also satisfied.

So, if $2rs \vert n$, we have an embedding $\Gamma$ of $K_n$ such that $D_r \x D_s \leq \TSG(\Gamma)$.

\subsection*{\bf Sufficiency for $D_3 \x D_3$ when $36 \vert (n-6)$.}  Suppose $n = 6 + 36k$.
Let $u$ be a point on $A$.
Then $\orb{u}_G$ consists of six points,
three on $A$, three on $B$.
We embed six vertices of $K_n$ as $\orb{u}_G$,
and the remaining $36k$ vertices as $X$.

We again verify that the conditions of the Edge Embedding Lemma are satisfied.
We need to do this only for pairs of fixed or interchanged vertices in $\orb{u}_G$
since we already verified the conditions for pairs in $X$
and there are no pairs of fixed on interchanged vertices
one of which is in $\orb{u}_G$ and the other in $X$.
Let $H$ denote the set of all elements of the form $\rho^i \s^j$, 
where $i, j \in \{1,2\}$.
There are exactly four such elements;
they are rotations by $\pm 2\pi/3$ about either $A$ or $B$,
and hence each fix exactly three vertices.
Let $u', u''$ be a pair of vertices in the orbit of $u$.
If they are pointwise fixed by distinct nontrivial elements $h_1, h_2 \in G$,
then we must have $u'$ and  $u''$ both in $A$ or both  in $B$,
$h_1^2= h_2 \in H$, and $\fix(h_1)=\fix(h_2)$.
Thus Condition~(1) is satisfied.
To satisfy Condition~(2), 
we let $A_{u' u''}$ be the arc in $A - \orb{u}_G$ or $B - \orb{u}_G$
whose boundary is $\{u', u''\}$.
Then the interior of any such arc is disjoint from the set of all embedded vertices
as well as from any other such arc.
Condition~(3) is satisfied since only elements in $H$ and elements conjugate to $\f$
fix a point in the interior of $A_{u' u''}$ or setwise fix a pair $\{u', u''\}$ bounding $A_{u' u''}$,
and all such elements take $A_{u' u''}$ to itself.
A pair $\{u', u''\} \subset \orb{u}_G$ is interchanged only by elements conjugate to $\a$, $\b$, or $\f$,
and any such element fixes at most one vertex.
Hence Condition~(4) is satisfied.
Also, any such element has nonempty fixed point set,
and any two such elements have distinct fixed point sets.
Hence Condition~(5) is satisfied.

So, if $36\vert (n-6)$,
then there exists an embedding $\Gamma$ of $K_n$
such that $D_3 \x D_3 \leq \TSG(\Gamma)$.

\subsection*{\bf Sufficiency for $\Z_3 \x D_3$ when $18 \vert (n-6)$.} Suppose $n = 18k + 6$, where $k$ is odd.
Let $G'=\orb{\rho,\s,\b} =\Z_3 \x D_3 \le G$.
Observe that $\orb{u}_{G'} = \orb{u}_G$.
We embed six vertices of $K_n$ as $\orb{u}_{G'} $,
and the remaining $18k$ vertices as the orbit under $G'$
of $k$ points $x_1, \cdots, x_k$ 
disjoint from the fixed point sets of all nontrivial elements of $G'$.
Then the conditions of the Edge Embedding Lemma
are satisfied by a similar argument as above  (but simpler since $\f \not \in G'$).  So there exists an embedding $\Gamma$ of $K_n$ such that $\Z_3 \x D_3 \leq \TSG(\Gamma)$.  

\subsection*{\bf Conclusion.}
So in each case we have an embedding $\Gamma$ of $K_n$ such that $\Z_r \x D_s \leq \TSG(\Gamma)$.  Since $\Z_r \x D_s$ is not a subgroup of $\so(3)$, the Complete Graph Theorem tells us that $\TSG(\Gamma)$ is a subgroup of $D_m \times D_m$ for some odd $m$.  Then, by the Subgroup Lemma, there is an embedding $\Gamma'$ of $K_n$ such that $\TSG(\Gamma') = \Z_r \x D_s$ (and, except in the case when $r =s = 3$ and $n$ is an odd multiple of $18$, an embedding $\Gamma''$ of $K_n$ such that $\TSG(\Gamma') = D_r \x D_s$).
\end{proof}

\small

\normalsize

\end{document}